\documentclass[aop]{imsart}

\RequirePackage{amsthm,amsmath,amsfonts,amssymb}
\RequirePackage[numbers]{natbib}

\startlocaldefs
\def\journal@name{}

\theoremstyle{plain}

\newtheorem{thm}{Theorem}
\newtheorem*{defi}{Definition}
\newtheorem{prop}[thm]{Proposition}
\newtheorem{lem}[thm]{Lemma}
\newtheorem{cor}[thm]{Corollary}
\newtheorem{rem}[thm]{Remark}
\newcommand{\RR}{\mathbb{R}}

\renewcommand{\SS}{\mathbb{S}}

\newcommand{\oo}{\mathbf{1}}

\newcommand{\dd}{\mathrm{d}}

\newcommand{\DR}{\mathrm{R}}

\newcommand{\PP}{\mathbb{P}}

\DeclareMathOperator{\var}{Var}

\DeclareMathOperator{\tr}{tr}

\newcommand{\lc}{\ensuremath \vartriangleleft}

\endlocaldefs

\begin{document}

\begin{frontmatter}
\title{Improved log-concavity for rotationally invariant measures of symmetric
convex sets}
\runtitle{Improved log-concavity for rotationally invariant measures}

\begin{aug}
%%%%%%%%%%%%%%%%%%%%%%%%%%%%%%%%%%%%%%%%%%%%%%%
%% Only one address is permitted per author. %%
%% Only division, organization and e-mail is %%
%% included in the address.                  %%
%% Additional information can be included in %%
\author[A]{\fnms{Dario} \snm{Cordero-Erausquin}\ead[label=e1]{dario.cordero@imj-prg.fr}}
\and
\author[B]{\fnms{Liran} \snm{Rotem}\ead[label=e2]{lrotem@technion.ac.il}}
%%%%%%%%%%%%%%%%%%%%%%%%%%%%%%%%%%%%%%%%%%%%%%
%% Addresses                                %%
%%%%%%%%%%%%%%%%%%%%%%%%%%%%%%%%%%%%%%%%%%%%%%
\address[A]{Institut de Mathématiques de Jussieu, Sorbonne Université, \printead{e1}}

\address[B]{Faculty of Mathematics, Technion -- Israel Institute of Technology, \printead{e2}}
\end{aug}

\begin{abstract}
We prove that the (B) conjecture and the Gardner-Zvavitch conjecture
are true for all log-concave measures that are rotationally invariant,
extending previous results known for Gaussian measures. Actually,
our result apply beyond the case of log-concave measures, for instance
to Cauchy measures as well. For the proof, new sharp weighted Poincar\'e inequalities
are obtained for even probability measures that are log-concave with
respect to a rotationally invariant measure. 
\end{abstract}



\end{frontmatter}
%%%%%%%%%%%%%%%%%%%%%%%%%%%%%%%%%%%%%%%%%%%%%%
%% Please use \tableofcontents for articles %%
%% with 50 pages and more                   %%
%%%%%%%%%%%%%%%%%%%%%%%%%%%%%%%%%%%%%%%%%%%%%%
%\tableofcontents

%%%%%%%%%%%%%%%%%%%%%%%%%%%%%%%%%%%%%%%%%%%%%%
%%%% Main text entry area:

\section{Introduction and main results}

Improved log-concavity inequalities under the assumption of symmetry
have become a central topic in the Brunn-Minkowski theory of convex
bodies, with several fascinating consequences and conjectures. Maybe
one of the first appearances of this phenomenon was the (B) inequality
established in \cite{Cordero-Erausquin2004} for a centered Gaussian
measure $\gamma$ on $\RR^{n}$. It states that for a symmetric convex
set $K\subset\RR^{n}$ (here symmetry means origin-symmetry, that
is $K=-K$) the function
\begin{equation}
t\to\gamma(e^{t}K)\qquad\textrm{ is log-concave on \ensuremath{\RR}. }\label{eq:b-ineq}
\end{equation}

A nonnegative function $m$ is said to be log-concave if $\left(-\log m\right)$
is a convex function with values in $\RR\cup\left\{ +\infty\right\} $.
The indicator of a convex set $C$ is a log-concave function and it
is common to denote the corresponding convex function by $\oo_{C}^{\infty}:=-\log\oo_{C}$, which is equal to $0$ on $K$ and to $+\infty$ outside. 

Property (\ref{eq:b-ineq}) was first conjectured by Banaszczyk and
popularized by Latała (\cite{Latala2002}). It found applications
outside Brunn-Minkowski theory, for instance in the setting of small
ball estimates in high dimensions, see \cite{Latala2005}, \cite{Klartag2007d}. 

A Borel measure $\mu$ on $\RR^{n}$ is said to satisfy the (B) property
if (\ref{eq:b-ineq}) holds for every symmetric convex set $K\subset\RR^{n}$
with $\mu$ in place of $\gamma.$ It is conjectured that every even
log-concave measure $\mu$, and by this we mean $\dd\mu(x)=e^{-V(x)}\dd x$
with $V$ convex and even, satisfies the (B) property. Prékopa's theorem
(\cite{Prekopa1971}) implies that every log-concave measure $\mu$
satisfies a Brunn-Minkowski inequality: For all convex sets $K,L$
the function 
\begin{equation}
[0,1]\ni t\to\mu((1-t)K+tL)\qquad\textrm{is log-concave.}\label{eq:BM0}
\end{equation}
This immediately implies that the function $s\to\mu(sK)$ is log-concave
on $\RR^{+}$. The conjecture is a strengthening of this property
under the extra assumption of symmetry. Saroglou (\cite{Saroglou2016})
showed that this (B) conjecture follows from another celebrated conjecture
for symmetric convex sets, namely the log-Brunn-Minkowski conjecture
 (\cite{Boroczky2012}). Combining the results of \cite{Saroglou2016}
and \cite{Boroczky2012} it follows that the conjecture holds in $\RR^{2}$.
Conversely, a certain strong form of the (B) conjecture will also imply
the log-Brunn-Minkowski conjecture (\cite{Saroglou2014}). 

In dimension $n\ge3$ very few examples of measures verifying the
(B) inequality were known and until now they all somehow relied on
the result for the Gaussian measure; These few non-Gaussian examples were
obtained by Eskenazis, Nayar and Tkocz in \cite{Eskenazis2018} and
will be discussed in Section \ref{sec:mixtures}. 

In a similar vein, a striking recent result of Eskenazis and Moschidis
(\cite{Eskenazis2021}) gives the following improvement of the log-concavity
(\ref{eq:BM0}) to $1/n$-concavity for a centered Gaussian measure
$\gamma$: if $K$ and $L$ are symmetric convex sets in $\RR^{n}$
and $\lambda\in[0,1],$ then 
\begin{equation}
\gamma((1-\lambda)K+tL)^{1/n}\ge(1-\lambda)\gamma(K)^{1/n}+\lambda\gamma(L)^{1/n}.\label{eq:GZ}
\end{equation}

This implies that the function $[0,1]\ni t\to\gamma((1-t)K+tL)^{1/n}$
is concave (because we work with convex sets). This property was conjectured
by Gardner and Zvavitch (\cite{Gardner2010}), and again it is connected
to several natural problems in the geometry of convex bodies. Note
that here the prototype of a measure satisfying this property is the
Lebesgue measure restricted to a convex set, by the Brunn-Minkowski
inequality. It is remarkable that the Gaussian measure also behaves
this way. One can ask whether every even log-concave measure satisfies
this Gardner-Zvavitch property. This conjecture will again be a corollary
of the log-Brunn-Minkowski conjecture (\cite{Livshyts2017a}), so
in particular it is known to hold in $\RR^{2}$. Building on earlier
ideas of Kolesnikov and Milman (\cite{Kolesnikov2018}, \cite{Kolesnikov2022}),
Kolesnikov and Livshyts (\cite{Kolesnikov2021}) proposed a convenient
spectral inequality that allowed them to show (\ref{eq:GZ}) with
exponent $\frac{1}{2n}$ (see Theorem \ref{thm:kolesnikov-livshyts}
below). This was improved to the optimal exponent $\frac{1}{n}$ in
\cite{Eskenazis2021}. In \cite{Colesanti2017} it was shown that
(\ref{eq:GZ}) holds for arbitrary rotation invariant log-concave
measures instead of $\gamma$, but only when $K$ and $L$ are small
perturbations of a ball. In \cite{Livshyts2021} Livshyts proved (\ref{eq:GZ})
for all even log-concave measures, but with the optimal exponent
$\frac{1}{n}$ replaced with a worse exponent $c_{n}=\frac{1}{n^{4+o(1)}}$. 

In the present paper, we show that there is nothing special about
the Gaussian measure and that properties (\ref{eq:b-ineq}) and (\ref{eq:GZ})
hold for every rotationally invariant log-concave measure on $\RR^{n},$
providing the first large class of measures on $\RR^{n}$ beyond Gaussian
measures satisfying the (B) conjecture and the Gardner-Zvavitch conjecture.
Actually, we will go beyond log-concave measures; for instance, we
will show that the Cauchy measures also satisfy these properties. 

Let us fix some notation in order to state our results. We consider
a finite dimensional Euclidean space $(\RR^{n},|\cdot|,\langle\cdot,\cdot\rangle)$.
For notational convenience, we assume we work with the standard structure
– note that the problems we study are affine invariant. A Borel measure
$\mu$ is rotationally invariant if $\mu(A)=\mu(RA)$ for every Borel
set $A$ and every linear map $R\in O(n)$. Since we are only considering
measures that are absolutely continuous with respect to the Lebesgue
measure $\dd x$, this means that we are considering measures of the
form 
\[
\dd\mu(x)=e^{-w(|x|)}\dd x
\]
for some function $w:\RR^{+}\to\RR\cup\left\{ +\infty\right\} $.
In this setting, we have
\begin{eqnarray*}
\mu\textrm{ is log-concave} & \Leftrightarrow & \ensuremath{x\to w(|x|)}\ensuremath{\textrm{ is convex on \ensuremath{\RR^{n}}}}\\
 & \Leftrightarrow & \ensuremath{w}\ensuremath{\textrm{ is increasing and convex on \ensuremath{\RR^{+}}}}\\
 & \Rightarrow & w\ensuremath{\textrm{ increasing and }t\to w(e^{t})}\textrm{ is convex on \ensuremath{\RR}.}
\end{eqnarray*}

This last condition will prove to be sufficient for establishing the
results (here and everywhere in the paper, "increasing" is understood in the weak sense, i.e. as non-decreasing). 
Note that for a smooth $w,$ the log-concavity of $\mu$
amounts to the conditions  $w'\ge0,\,w''\ge0$, whereas our weaker
assumption is equivalent to $w'\ge0,\,sw''(s)+w'(s)\ge0$. Also, unless
$w$ is constant (which is not a situation of interest) we will always
have that $w(t)\to+\infty$ as $t\to+\infty.$

Throughout the paper, the notion of "being log-concave with respect to a measure" will play an important role, and so let us introduce a specific notation for that. 

\begin{defi}[Log-concavity preorder $\lc$ on measures]
Given two Borel measures $\mu,\nu$ on $\RR^n$, we say that $\nu$ is log-concave with respect to $\mu$, and write $\nu\lc\mu$ if $\nu$ has a log-concave density with respect to $\mu$, that is
$$\dd\nu(x)=e^{-v(x)}\, \dd \mu(x)$$
with $v:\RR^n\to\RR\cup\{+\infty\}$ convex. 
\end{defi}
Of course, for any measure $\mu$ and any constant $c>0$ we have $c\, \mu\lc\mu$; the central example is $\mu_K\lc\mu$ where $\mu_K$ is the restriction of a measure $\mu$ to a convex set $K\subset\RR^n$ defined by $\mu_K(A)=\mu(A\cap K)$, possibly renormalized into a probability measure when $\mu(K)<\infty$. In Bakry-Emery comparison geometry terms, one could say that $\nu\lc\mu$ means that  $\nu$ is \emph{more curved} than $\mu$, although one must pay attention that no curvature is \emph{a priori} imposed on $\mu$. It is nonetheless natural to ask whether some (well chosen) functional inequality is valid for the whole class of such $\nu$'s. Our main contribution in this direction for rotationally invariant measures will be Theorem~\ref{thm:weighted-poincare} below.

We begin with the (B) conjecture. Actually, we are able to extend
the same \emph{strong} form that was established for the Gaussian
measure to every log-concave (and beyond) rotationally invariant measure. 
\begin{thm}
\label{thm:B-result}Let $w:(0,\infty)\to(-\infty,\infty]$ be an
increasing function such that $t\to w(e^{t})$ is convex, and let
$\mu$ be the measure on $\RR^{n}$ with density $e^{-w\left(\left|x\right|\right)}$.
Then for every symmetric convex body $K\subseteq\RR^{n}$, and every
\emph{symmetric} matrix $A$, the function 
\[
t\mapsto\mu\left(e^{tA}K\right)
\]
 is log-concave. 
\end{thm}

Let us mention that we will actually prove the following more general
statement: under the same assumptions  on $w$, if $v:\RR^{n}\to\RR\cup\left\{ +\infty\right\} $
is an even convex function then the function 
\begin{equation}
t\to\int_{\RR^{n}}e^{-v(e^{tA}x)-w\left(\left|x\right|\right)}\dd x\qquad\textrm{ is log-concave on \ensuremath{\RR}. }\label{eq:functional_B_property}
\end{equation}
The theorem corresponds to the choice $v=1_{K}^{\infty}$,
%\[ 
%v(x)=1_{K}^{\infty}(x) := \begin{cases}
%0 & x \in K \\
%+\infty & \text{otherwise}
%\end{cases}
%\] 
replacing $t$ by $-t$.
%, which preserves log-concavity). 
This ``functional''
version of the (B) property was previously studied in \cite{Cordero-Erausquin2020}. 

Note that the functions $w(t)=w_{p}(t):=t^{p}/p$ satisfy the assumptions
of the theorem for all $p>0$. Hence the corresponding measures 
\[
\dd\mu_{p}=e^{-\left|x\right|^{p}/p}\dd x
\]
all have the strong (B) property. Taking $w=\oo_{[0,1]}^{\infty}$
(that is $p\to+\infty)$ we see that the uniform measure on $B_{2}^{n}$
also has the strong (B) property. Recall we are free to pick any Euclidean
structure, or equivalently we can work with measures on $\RR^{n}$
of the form 
\begin{equation}
\dd\mu(x)=e^{-w(\langle Cx,x\rangle)}\:\dd x\label{eq:quadratic}
\end{equation}
where $C$ is a symmetric positive matrix; one just has  to be careful
when stating the strong (B) property as in this case the symmetry condition
on the matrix $A$ is that $CA=A^{\ast}C$. For the classical (B)
property ($A={\rm Id}$) there is no issue here and in particular
the uniform measure on an ellipsoid satisfies the (B) inequality.
Note also that if $\mathcal{E}$ is an ellipsoid we may use (\ref{eq:functional_B_property})
with $A={\rm Id}$, $w=\oo_{[0,1]}^{\infty}$ and the norm associated
to $\mathcal{E}$. Performing the change of variables $y=e^{tA}x$
(whose Jacobian is log-linear and hence immaterial), we derive the
following corollary.

\begin{cor}
Let  $\nu$ be an arbitrary even log-concave measure
on $\RR^{n}$ and $\mathcal{E}\subset \RR^n$ be an ellipsoid. Then we have that 
%\begin{equation}
$$t\to\nu(e^{t}\mathcal{E})\quad\textrm{is log-concave on \ensuremath{\RR}. }\label{eq:dilate_ellipsoid}
$$
%\end{equation}
\end{cor}

It is also worth mentioning that by approximation, our results apply
to degenerate nonnegative quadratic forms as well, that is to the
case where the matrix $C$ in (\ref{eq:quadratic}) is degenerate.
For instance we can consider measures of the form $e^{-w(\sqrt{x_{1}^{2}+\ldots x_{k}^{2}})}\,\dd x$
with $k\le n$. 

Let us give some further examples of non-log-concave measures that
satisfy our assumptions and for which our results hold. For instance,
by taking $w(t)=a\log t+\tilde{w}(t)$, for any $a\ge0$ and $\tilde{w}$
satisfying our assumptions (possibly $\tilde{w}\equiv0)$, we can
consider measures of the form
\[
\dd\mu(x)=|x|^{-a}e^{-\tilde{w}(|x|)}\dd x.
\]

It is reasonable to impose local integrability (around zero) of the
density, that is $0\le a<n$, for if not the measure of every symmetric
convex body  is $+\infty,$ and the result is therefore less interesting.
We can also take for instance $w(t)=a\log(1+t^{b})$ for any $a,b\ge0$
and work with measures of the form
\[
\dd\mu(x)=(1+|x|^{b})^{-a}\:\dd x,
\]
which include Cauchy-type measures on $\RR^{n}$.

Note that our condition on the density $e^{-w(|x|)}$ of the rotationally measure $\mu$ is stable under products, because the condition on the corresponding function $w$ is stable under additions. In particular we can
replace $\dd x$ in the previous examples by a suitable rotationally invariant measure. For example, we can restrict these measures to a centered Euclidean ball. 

Let us now comment on the proof of the (B) inequality. It is well known that taking second derivatives reduces Brunn-Minkowski
type inequalities to spectral inequalities for some differential operator.
The proof of Theorem \ref{thm:B-result} follows a scheme similar to
\cite{Cordero-Erausquin2004}, which handled the Gaussian case $\mu_{2}$
by establishing a connection with a ``second eigenvalue problem''
associated to measures that are log-concave with respect to $\mu_{2}$.
We will reduce the problem to a spectral inequality of Brascamp–Lieb
type, in an improved form for even functions. By examining the Gaussian
case one could seek an improvement in the constant of a ``classical''
spectral inequality. However, we believe this would not be the right
way to go (see the remark at the end of Section \ref{sec:the-B-property}).
Moreover, already for our rotationally invariant measure $\mu$, we
do not know the exact whole spectrum (unlike the gaussian case), and
we need anyway to work with measures that are log-concave with respect
to $\mu.$ We will instead establish the following sharp spectral inequality,
from which the result follows:
\begin{thm}
\label{thm:BL-even}Let $w:[0,\infty)\to\RR$ be a $C^{2}$-smooth
increasing function such that $t\mapsto w(e^{t})$ is convex. Define
$W:\RR^{n}\to\RR$ by $W(x)=w\left(\left|x\right|\right)$ and let $\nu$ be an even probability measure with $\nu\lc e^{-W(x)}\, \dd x$. 

Then for every \emph{even} $C^{1}$ function $f:\RR^{n}\to\RR$ such
that $\int\left|f\right|^{2}\dd\nu<\infty$ we have
\[
\var_{\nu}f\le\int\left\langle \left(\nabla^{2}W+\frac{w'\left(\left|x\right|\right)}{\left|x\right|}\,{\rm Id}\right)^{-1}\nabla f,\nabla f\right\rangle \dd\nu.
\]
\end{thm}

Here and for the rest of the paper, we omit the dependence on the variable except where necessary. Also, as we will see in Section \ref{sec:BL}, the matrix $\nabla^{2}W+\frac{w'\left(\left|x\right|\right)}{\left|x\right|}\,{\rm Id}$
is always positive \emph{semi}-definite, but it can be singular (formulation should then be adapted, as explained in Remark~\ref{rem:dual} below), although in applications we can assume by approximation that
this situation does not arise. 

One can check that equality holds in Theorem \ref{thm:BL-even} when
 $\dd\nu=\frac{e^{-W(x)}}{{\it \int e^{-W}}}\dd x$, with $\int e^{-W}<\infty$, 
and $f(x)=\left\langle \nabla W,x\right\rangle =w'\left(\left|x\right|\right)\left|x\right|$
. This can be seen without explicit computations by inspecting the
use of Theorem \ref{thm:BL-even} in the proof of Theorem \ref{thm:B-result}
in Section \ref{sec:the-B-property} and using the fact that for $K=\RR^{n}$
the function $t\mapsto\mu\left(e^{t}K\right)$ is constant and thus
log-linear. 

In the case of the Gaussian measure $\mu_{2}$ (that is $w(t)=t^{2}/2)$,
this inequality reduces to the fact that for an even measure $\nu\lc\mu_{2},$ one has $\var_{\nu}f\le\frac{1}{2}\int|\nabla f|^{2}\:\dd\nu$
for every even smooth $f$. 

The matrix $\nabla^{2}W+\frac{w'\left(\left|x\right|\right)}{\left|x\right|}\,{\rm Id}$
is a rank one perturbation of a scalar matrix (see (\ref{eq:matrix-explicit}) %%R1, comment 3
below), so we can compute its inverse explicitly. The result is that
under the assumptions of Theorem~\ref{thm:BL-even} we have the inequality
\[
\var_{\nu}f\le\int\left(\frac{\left|x\right|}{2w'\left(\left|x\right|\right)}\left|\nabla f\right|^{2}-\frac{\left|x\right|w''\left(\left|x\right|\right)-w'\left(\left|x\right|\right)}{2\left|x\right|w'\left(\left|x\right|\right)\left(\left|x\right|w''\left(\left|x\right|\right)+w'\left(\left|x\right|\right)\right)}\left\langle \nabla f,x\right\rangle ^{2}\right)\dd\nu.
\]
For example, taking $w(t)=w_{p}(t)=t^{p}/p$ we see that for an even  $\nu\lc\mu_{p}$ we have
\[
\var_{\nu}f\le\int\left(\frac{1}{2}\left|x\right|^{2-p}\left|\nabla f\right|^{2}-\frac{p-2}{2p}\cdot\frac{\left\langle \nabla f,x\right\rangle ^{2}}{\left|x\right|^{p}}\right)\dd\nu.
\]
 Using the trivial bounds $0\le\left\langle \nabla f,x\right\rangle ^{2}\le\left|\nabla f\right|^{2}\left|x\right|^{2}$
we can deduce the less precise but more elegant inequality 
\[
\var_{\nu}f\le\max\left\{ \frac{1}{p},\frac{1}{2}\right\} \cdot\int\left|x\right|^{2-p}\left|\nabla f\right|^{2}\dd\nu.
\]
This inequality is still sharp when $\nu=c\cdot\mu_{p}$ for $0<p\le2$
and $c$ a normalization constant, with equality when $f(x)=|x|^p$. Similarly,
taking $w_{C}(t)=a\cdot\log\left(1+t^{2}\right)$ we see that when
$\nu$ is log-concave with respect to the Cauchy-type measure $\dd\mu_C=\frac{1}{\left(1+\left|x\right|^{2}\right)^{a}}\dd x$
we obtain the inequalities
\[
\var_{\nu}f\le\frac{1}{4a}\int\left(1+\left|x\right|^{2}\right)\left(\left|\nabla f\right|^{2}+\left\langle \nabla f,x\right\rangle ^{2}\right)\dd\nu\le\frac{1}{4a}\int\left(1+\left|x\right|^{2}\right)^{2}|\nabla f|^{2}\,\dd\nu.
\]
 Again, both of these inequalities are sharp when $\nu$ is the (normalized) reference measure $\mu_C$.
This last inequality is similar in spirit to a result of Bobkov and
Ledoux (\cite{Bobkov2009}) for Cauchy measures, which is only sharp
up to a universal constant but holds for non-even functions; this was recently sharpened in \cite{Bobkov2022} in the case $a=n$.   

In the Gaussian case, the above-mentioned inequality $\var_{\nu}f\le\frac{1}{2}\int|\nabla f|^{2}\:\dd\nu$
for $f$ even was at the heart of the argument in \cite{Cordero-Erausquin2004}.
It was established using an $L^{2}$ argument with a Bochner integration
by parts (a second argument using Cafferelli's contraction property
was also given). The argument used the following classical Poincaré inequality,
which follows from the variance Brascamp–Lieb inequality (\cite{Brascamp1976})
or the Bakry–Emery criterion (\cite{Bakry1985}): For a probability measure $\nu$
with $\nu\lc\mu_2$, one has for every
smooth $h$ that $\var_{\nu}h\le\int|\nabla h|^{2}\:\dd\nu.$ For
our general $\mu$ this inequality needs to be replaced by a \emph{weighted}
Poincaré inequality that appears to be new, even in the simple case
of of the measure $e^{-w(|x|)}\dd x$ for which it is sharp. In fact we will only prove
such an inequality when the function is \emph{odd}, which is good
enough for our purposes. 
\begin{thm}
\label{thm:weighted-poincare}Let $w:(0,\infty)\to\RR$ be $C^{1}$-smooth
and increasing and let $\nu$ be an even finite measure on $\RR^n$ with  $\nu\lc e^{-w\left(\left|x\right|\right)}\dd x$.

Then for every $C^{1}$-smooth and \textbf{odd} function $h:\RR^{n}\to\RR$
%with $\nabla h\in L^{2}(\nu)$
we have 
\[
\int\frac{w'\left(\left|x\right|\right)}{\left|x\right|}h^{2}\dd\nu\le\int\left|\nabla h\right|^{2}\dd\nu.
\]
\end{thm}

As we will see after the proof, equality holds in Theorem \ref{thm:weighted-poincare}
 when $\dd\nu=e^{-w(\left|x\right|)}\dd x$, and $h(x)=\left\langle x,\theta\right\rangle $
for a fixed $\theta\in\RR^{n}$. 

In the two cases of interest from before $w=w_{p}$ and $w=w_{C}$,
Theorem \ref{thm:weighted-poincare} reduces to the inequalities 
\[
\int\left|x\right|^{p-2}h^{2}\dd\nu\le\int\left|\nabla h\right|^{2}\dd\nu\quad\text{and}\quad\int\frac{h^{2}}{1+\left|x\right|^{2}}\dd\nu\le\frac{1}{2a}\int\left|\nabla h\right|^{2}\dd\nu
\]
respectively. Both of these inequalities are sharp when $\text{\ensuremath{\dd}\ensuremath{\ensuremath{\nu}=\ensuremath{e^{-|x|^{p}/p}\dd}x}}$
and $\dd\nu=\frac{1}{(1+|x|^{2})^{a}}\dd x$ with $a>n/2$ (condition for finiteness), respectively,
with equality for linear functions. 

It turns out that our weighted Poincaré inequality above (Theorem~\ref{thm:weighted-poincare})
allows us to solve the problem of the Gardner-Zvavitch conjecture
for rotationally invariant measures, with the same condition as
for the (B) inequality.

\begin{thm}
\label{thm:GZ}Let $w:[0,\infty)\to(-\infty,\infty]$ be an increasing
function such that $t\mapsto w(e^{t})$ is convex and let $\mu$ be
the measure on $\RR^{n}$ with density $e^{-w\left(\left|x\right|\right)}$.
Then for every symmetric convex bodies $K,L\subset\RR^{n}$ and $\lambda\in[0,1],$
\[ %%R1, comment 6
\mu((1-\lambda)K+\lambda L)^{1/n}\ge(1-\lambda)\mu(K)^{1/n}+\lambda\mu(L)^{1/n}
\]
 
\end{thm}

As before, our result includes all rotationally invariant log-concave
measures but applies also beyond this class – see the examples given
above, such as Cauchy type measures. 

There is some surprising phenomenon here that we would like to outline. It was observed by Borell (\cite{Borell1974}, \cite{Borell1975}) that if a measure $\mu$ with density $f$ on $\RR^n$ satisfies any kind of Brunn-Minkowski inequality, even in the weakest form $\mu((1-\lambda)K+\lambda L)\ge \min\{\mu(K),\mu(L)\}$ for every convex sets $K,L$ and $\lambda\in(0,1)$, then $f$ must satisfy some concavity property. It follows from Borell's observation that  the family of measures $\dd\mu_C=\frac{1}{\left(1+\left|x\right|^{2}\right)^{a}}\dd x$
with $a>0$ satisfy \emph{no} Brunn-Minkowski inequality when $2a<n$.  However,  when restricted to \emph{symmetric} convex sets, all these measures satisfy the strong Brunn-Minkowski inequality given by the previous theorem.

The rest of the paper is devoted to the proofs of the main results
and some further comments and extensions. In the next Section we prove
the weighted Poincaré inequality (Theorem \ref{thm:weighted-poincare}).
Then, in Section 3 we will use it to establish our spectral estimate
of Brascamp-Lieb type for even functions (Theorem \ref{thm:BL-even}).
We show in Section 4 that this spectral estimate in turn implies the
strong (B) inequality (Theorem \ref{thm:B-result}). In Section 5,
we give the proof of the dimensional Brunn-Minkowski inequality (Theorem
\ref{thm:GZ}). In the final Section 6, following an argument of \cite{Eskenazis2018}
we explain how to extend the (B) inequality to mixtures of rotationally
invariant measures, thus providing new examples of measures satisfying
this property.

\section{Weighted Poincaré inequalities}

In this section we give the proof of Theorem \ref{thm:weighted-poincare}.
We will proceed by integration in polar coordinates: for an integrable
or nonnegative function $F$ on $\RR^{n}$,
\[
\int F(x)\dd x=c_{n}\int_{\SS^{n-1}}\int_{0}^{\infty}F(r\theta)\,r^{n-1}\dd r\dd\theta,
\]
where $\dd\theta$ refers to the usual normalized measure on the sphere
$\SS^{n-1}=\left\{ x:\ \left|x\right|=1\right\} $. Therefore we will
need two Poincaré type inequalities, one for the spherical part and
one for the radial part. 

In order to treat the spherical part, we will need the following weighted
Poincaré inequality on the sphere that is a particular case of a general
result of Kolesnikov and Milman (\cite{Kolesnikov2018}), as we shall
see. 
\begin{prop}
\label{prop:spherical-poincare}Let $v:\RR^{n}\to\RR$ be a convex
$C^{1}$ function and let $\mu$ be the measure on $\SS^{n-1}$ with
density $e^{-v}$. Then for every $C^{1}$ function $g:\SS^{n-1}\to\RR$
with $\int_{\SS^{n-1}}g\,\dd\mu=0$ one has 
\[
\int_{\SS^{n-1}}\left(n-1-\DR v\right)g^{2}\dd\mu\le\int_{\SS^{n-1}}\left|\nabla_{\SS}g\right|^{2}\dd\mu.
\]
 
\end{prop}

Here and after, we use the notation $\DR v(x)=\left\langle \nabla v(x),x\right\rangle $
for the radial derivative and $\nabla_{\SS}g$ for the spherical gradient
of $g$. 

Indeed, generalizing a result of Colesanti (\cite{Colesanti2008}),
Kolesnikov and Milman proved the following very general inequality:
\begin{thm}[\cite{Kolesnikov2018}]
\label{thm:kolesnikov-milman}Let $\left(M,g\right)$ be a compact,
smooth, complete, connected and oriented $n$-dimensional Riemannian
manifold with boundary $\partial M$. Let $\dd \mu=e^{-v}\dd\text{Vol}_{M}$
be a measure on $M$, where $v:M\to\RR$ is $C^{2}$-smooth. 

Assume $(M,g,\mu)$ satisfies the $CD(0,N)$ condition for some $N$
such that $\frac{1}{N}\in[-\infty,\frac{1}{n}]$, and that $\mathrm{II}_{\partial M} \succ 0$.
Then for every $f\in C^{1}\left(\partial M\right)$ we have 
\[
\int_{\partial M}H_{\mu}f^{2}\dd\mu_{\partial M}-\frac{N-1}{N}\cdot\frac{\left(\int_{\partial M}f\dd\mu_{\partial M}\right)^{2}}{\mu(M)}\le\int_{\partial M}\left\langle \mathrm{II}_{\partial M}^{-1}\nabla_{\partial M}f,\nabla_{\partial M}f\right\rangle \dd\mu_{\partial M}.
\]
 
\end{thm}

To explain the notation of the theorem, we say that $\left(M,g,\mu\right)$
satisfies the $CD(0,N)$ condition if 
\[
\mathrm{Ric}_{g,\mu}:=\mathrm{Ric}_{g}+\nabla^{2}v-\frac{1}{N-n}\dd v\otimes\dd v\succeq0
\]
as a $2$-tensor, where $\mathrm{Ric}_{g}$ denotes the classical
Ricci curvature. Furthermore $\mathrm{II}_{\partial M}$ denotes the
second fundamental form, and $H_{\mu}(x)=\tr\left(\mathrm{II}_{\partial M}(x)\right)-\left\langle \nabla v(x),\nu(x)\right\rangle $
denotes the weighted mean curvature of $\partial M$ at $x\in\partial M$,
where $\nu(x)$ is the outer unit normal to $\partial M$ at $x$. 

To see why Proposition \ref{prop:spherical-poincare} follows from
Theorem \ref{thm:kolesnikov-milman} we simply choose $M=B_{2}^{n}\subset\RR^{n}$,
the unit Euclidean ball with the standard Euclidean metric. By approximation
we may assume $v$ is $C^{2}$. Then for $\dd\mu=e^{-v}\dd x$ the
weighted manifold $\left(M,g,\mu\right)$ satisfies the $CD(0,\infty)$
condition since 
\[
\mathrm{Ric}_{g,\mu}=0+\nabla^{2}v+0=\nabla^{2}v\succeq0.
\]
 Moreover in this case $\mathrm{II}_{\SS^{n-1}}(x)$ is given by the
standard inner product on $\RR^{n}$ for all $x\in\SS^{n-1}=\partial B_{2}^{n}$,
so 
\[
H_{\mu}(x)=n-1-\left\langle \nabla v(x),x\right\rangle =n-1-\DR v(x).
\]
Plugging this into Theorem \ref{thm:kolesnikov-milman}, one obtains
Proposition \ref{prop:spherical-poincare}. 

Let us comment a bit more on this result. Kolesnikov and Milman obtained
their inequality using a general Reilly-type integration by parts
formula for the solution $u$ of the problem $\Delta_{g}u-\langle\nabla u,\nabla v\rangle\equiv\frac{1}{\mu(M)}\int_{\partial M}f\dd\mu_{\partial M}$
in the interior of $M$ and the normal derivative of $u$ on $\partial M$
equal to $f$. However, when $M$ is a convex body in $\RR^{n}$ this
inequality can be derived in a more elementary way by differentiating
the Brunn-Minkowski inequality (\ref{eq:BM0}) for the log-concave
measure $\dd\mu=e^{-v}\dd x$ – see \cite{Kolesnikov2018} (in particular
Theorem 6.6) and \cite{Kolesnikov2021} (in particular Proposition
3.2). When $v=0$ this is exactly what was done in \cite{Colesanti2008},
but it is absolutely crucial for us to have the correct inequality
for the weight $e^{-v}$. 

The second ingredient we will need for the proof is the following
one dimensional lemma. 
\begin{lem}
\label{lem:1d-poincare}Let $w,v:[0,\infty)\to\RR$ be continuous
functions and $C^{1}$ on $(0,\infty)$. Let $f$ be a $C^{1}$ function
on $[0,\infty)$ which is compactly supported (for simplicity) and
satisfies $f(0)=0$. Then for every $\alpha\ge0$ we have 
\[
\int_{0}^{\infty}\frac{w'}{t}f^{2}t^{\alpha}e^{-w-v}\dd t\le\int_{0}^{\infty}\left(\left(f'\right)^{2}+\alpha\cdot\left(\frac{f}{t}\right)^{2}-v'\frac{f^{2}}{t}\right)t^{\alpha}e^{-w-v}\dd t.
\]
\end{lem}

\begin{proof}
Since $f$ is $C^1$-smooth and $f(0)=0$ we may write $f(t)=t\cdot g(t)$ for a function $g$ continuous
on $[0,\infty)$, $C^{1}$ on $(0,\infty)$ and compactly supported.
It follows using integration by parts, since boundary terms vanish,
that 
\begin{align*}
\int_{0}^{\infty}\frac{w'}{t}f^{2}t^{\alpha}e^{-w-v}\dd t & =\int_{0}^{\infty}w'g^{2}t^{\alpha+1}e^{-w-v}\dd t=-\int_{0}^{\infty}\left(g^{2}t^{\alpha+1}e^{-v}\right)\left(e^{-w}\right)^{\prime}\dd t.\\
 & =\int_{0}^{\infty}\left(g^{2}t^{\alpha+1}e^{-v}\right)^{\prime}e^{-w}\dd t\\
 & =\int_{0}^{\infty}\left(2tgg'+(\alpha+1)g^{2}-v'g^{2}t\right)t^{\alpha}e^{-w-v}\dd t.
\end{align*}
 On the other hand, we have for the right-hand side:
\begin{multline*}
\int_{0}^{\infty}\left(\left(f'\right)^{2}+\alpha\cdot\left(\frac{f}{t}\right)^{2}-v'\frac{f^{2}}{t}\right)t^{\alpha}e^{-w-v}\dd t \\ 
\begin{aligned}
&=\int_{0}^{\infty}\left(\left(g+tg'\right)^{2}+\alpha\cdot g^{2}-v'tg^{2}\right)t^{\alpha}e^{-w-v}\dd t   \\
  &=\int_{0}^{\infty}\left(g^{2}+2tgg'+t^{2}\left(g'\right)^{2}+\alpha g^{2}-v'tg^{2}\right)t^{\alpha}e^{-w-v}\dd t.  \qquad
  \end{aligned}
\end{multline*}
 Comparing the two expressions we see that the difference between
the right hand side and the left hand side is exactly $\int_{0}^{\infty}\left(g'\right)^{2}t^{\alpha+2}e^{-w-v}\dd t$,
which is clearly non-negative. 
\end{proof}
We are now ready to prove Theorem \ref{thm:weighted-poincare}:
\begin{proof}[Proof of Theorem \ref{thm:weighted-poincare}]
Our finite measure $\nu$ is of the form $\dd\nu(x)=e^{-v(x)-w(|x|)}\, \dd x$ with $v:\RR^n\to\RR\cup\left\{ +\infty\right\}$ convex. We can assume that our $h$ satisfies $\nabla h \in L^2(\nu)$ since otherwise there is nothing to prove. 

We begin with some standard approximation arguments. First let us
note that we can assume that $h\in L^{2}(\nu)$. Actually, we can
assume that $h$ is bounded. Indeed, let us introduce for every $k\in\mathbb{N}^{\ast}$
a $C^{1}$ smooth non-decreasing odd function $R_{k}:\RR\to\RR$ such
that $R_{k}(t)=t$ for $t\in[0,k]$, $R_{k}(t)\equiv k+1$ for $t\ge k+2$,
$R_{k}(t)\le t$ and $R_k^{\prime}(t)\le1$ for every $t\in\RR^{+}.$ Then the %%R1, comment 8
functions $h_{k}:=R_{k}(h)$ satisfy $h_{k}=h$ on the open set $\{|h|<k\}$,
$|h_{k}|\uparrow|h|$ and $|\nabla h_{k}|\stackrel{\le}{\longrightarrow}|\nabla h|$.
Hence by monotone and dominated convergence, respectively,  we can
pass from the bounded functions $h_{k}$ to $h$ in our inequality. 

Next we reduce to the case that $h$ is compactly supported. The classical
argument is to introduce a smooth and radially decreasing function $\chi$ on $\RR^{n}$ with %%R1, comment 9
values in $[0,1]$ that is compactly supported and equals to $1$
in a neighborhood of $0$ and to set $\chi_{k}(x):=\chi(x/k)$. Then
$\chi_{k}\uparrow1$ and $|\nabla\chi_{k}|\le C/k$ for some constant
$C>0$. On one hand we have $\frac{w'\left(\left|x\right|\right)}{\left|x\right|}(h\chi_{k})^{2}\uparrow\frac{w'\left(\left|x\right|\right)}{\left|x\right|}h^{2}$
and, on the other hand,
\begin{align*}
\int|\nabla(h\chi_{k})|^{2}\,\dd\nu & =\int|\nabla h|^{2}\chi_{k}^{2}\,\dd\nu+2\int h\chi_{k}\langle\nabla h,\nabla\chi_{k}\rangle\,\dd\nu+\int h^{2}|\nabla\chi_{k}|^{2}\,\dd\nu\\
 & \le\int|\nabla h|^{2}\,\dd\nu+\frac{2C}{k}\sqrt{\int h^{2}\,\dd\nu\int|\nabla h|^{2}\,\dd\nu}+\frac{C^{2}}{k^{2}}\int h^{2}\,\dd\nu
\end{align*}
and this upper bound converges to $\int|\nabla h|^{2}\dd\nu$, as wanted. 

Finally, we approximate $w$ and $v$. By replacing $w(t)$ with $\max(w(t),-k)$
and invoking monotone convergence as $k\to\infty$, we can assume
that $w$ is continuous on the closed ray $[0,\infty)$ and $C^{1}$
on $(0,\infty)$ except maybe at one point, which is irrelevant. By
standard approximation we may also assume without loss of generality
that $v$ is smooth. 

These remarks being made, we compute the integrals for our compactly
supported function $h$ using polar coordinates. We obtain 
\[
\int\frac{w'\left(\left|x\right|\right)}{\left|x\right|}h^{2}\dd\nu=c_{n}\int_{\SS^{n-1}}\int_{0}^{\infty}\frac{w'(r)}{r}h^{2}(r\theta)r^{n-1}e^{-w(r)-v(r\theta)}\dd r\dd\theta.
\]
For a fixed $\theta\in\SS^{n-1}$ we will now apply Lemma \ref{lem:1d-poincare}
with $f_{\theta}(r)=h(r\theta)$, $v_{\theta}(r)=v(r\theta)$ and
$\alpha=n-1$. Note that $v_{\theta}'(r)=\left\langle \nabla v\left(r\theta\right),\theta\right\rangle =\frac{1}{r}\DR v\left(r\theta\right)$.
Therefore we can bound our integral by 
\[
c_{n}\int_{S^{n-1}}\int_{0}^{\infty}\left(\underbrace{\left\langle \nabla h(r\theta),\theta\right\rangle ^{2}}_{\mathrm{I}}+\underbrace{(n-1)\left(\frac{h(r\theta)}{r}\right)^{2}-\frac{1}{r}\DR v(r\theta)\cdot\frac{h(r\theta)^{2}}{r}}_{\mathrm{II}}\right)r^{n-1}e^{-w(r)-v(r\theta)}\dd r\dd\theta.
\]
We will leave term $\mathrm{I}$ as is for now. In order to bound
term $\mathrm{II}$ we change the order of integration: 
\begin{align*}
\mathrm{II} & =c_{n}\int_{0}^{\infty}\int_{S^{n-1}}(n-1-\DR v(r\theta))\left(\frac{h(r\theta)}{r}\right)^{2}r^{n-1}e^{-w(r)-v(r\theta)}\dd\theta\dd r\\
 & =c_{n}\int_{0}^{\infty}r^{n-1}e^{-w(r)}\left(\int_{\SS^{n-1}}(n-1-\DR v(r\theta))\left(\frac{h(r\theta)}{r}\right)^{2}e^{-v(r\theta)}\dd\theta\right)\dd r
\end{align*}
We now apply Proposition \ref{prop:spherical-poincare} to the inner
integral, with $v_{r}(\theta)=v(r\theta)$, $\mu_r = e^{-v_r} \dd \theta $ and
$g_{r}(\theta)=\frac{h(r\theta)}{r}$.
Note that since $v_r$ is even and $g_r$ is odd we indeed have $\int_{\SS^{n-1}} g_r \dd \mu_r = 0$ 
(this is in fact the only place where we use the fact that $h$ is odd). %%R1, comment 10
Also note that $\DR v_{r}(\theta)=\DR v(r\theta)$ and $\nabla_{\SS}g_{r}(\theta)=\nabla_{\SS}h(r\theta)$,
where for a function $h:\RR^{n}\to\RR$ the notation 
\[
\nabla_{\SS}h(x)=\nabla h(x)-\left\langle \nabla h(x),\frac{x}{\left|x\right|}\right\rangle \cdot\frac{x}{\left|x\right|}
\]
 denotes the tangential part of the gradient of $h$. We may therefore
apply the proposition and  conclude that 
\[
\mathrm{II}\le c_{n}\int_{0}^{\infty}r^{n-1}e^{-w(r)}\int_{\SS^{n-1}}\left|\nabla_{\SS}h(r\theta)\right|^{2}e^{-v(r\theta)}\dd\theta\dd r.
\]
 Using this estimate for $\mathrm{II}$ we conclude that 
\begin{align*}
\int\frac{w'\left(\left|x\right|\right)}{\left|x\right|}h^{2}\dd\nu & \le c_{n}\int_{\SS^{n-1}}\int_{0}^{\infty}\left(\left\langle \nabla h(r\theta),\theta\right\rangle ^{2}+\left|\nabla_{\SS}h(r\theta)\right|^{2}\right)r^{n-1}e^{-w(r)-v(r\theta)}\dd r\dd\theta\\
 & =c_{n}\int_{\SS^{n-1}}\int_{0}^{\infty}\left|\nabla h(r\theta)\right|^{2}r^{n-1}e^{-w(r)-v(r\theta)}\dd r\dd\theta\\
 & =\int\left|\nabla h(x)\right|^{2}e^{-w\left(\left|x\right|\right)-v(x)}\dd x=\int\left|\nabla h\right|^{2}\dd\nu,
\end{align*}
 completing the proof of Theorem \ref{thm:weighted-poincare}. 
\end{proof}
We remark that the proof above strongly uses the fact that $h$ is odd to deduce that the functions
$g_r e^{-v_r}$ are all centered. It is possible that Theorem \ref{thm:weighted-poincare} is true under 
a weaker assumption on $h$, but at the moment we do not know how to address this interesting question. 

To conclude this section let us prove that when $v=0$, that is $\dd\nu=e^{-w\left(\left|x\right|\right)}\dd x$,
we have equality in Theorem \ref{thm:weighted-poincare} for every
linear function $h(x)=\left\langle x,\theta\right\rangle $. By homogeneity
and rotation invariance it is enough to consider the function $h(x)=x_{1}$.
Formally, the result follows by integration by parts:
\[
\int\frac{w'\left(\left|x\right|\right)}{\left|x\right|}h^{2}\dd\nu=-\int\partial_{1}\left(e^{-w\left(\left|x\right|\right)}\right)x_{1}\dd x=\int1\cdot e^{-w\left(\left|x\right|\right)}\dd x=\int\left|\nabla h\right|^{2}\dd\nu.
\]
 To check this rigorously, introduce $A_{\epsilon,R}=\left\{ x\in\RR^{n}:\ \epsilon<\left|x\right|<R\right\} $
for $0<\epsilon<R<\infty$. Then using polar coordinates and integration
by parts we have 
\begin{align}
\int_{A_{\epsilon,R}}\frac{w'\left(\left|x\right|\right)}{\left|x\right|}h^{2}\dd\nu & =\frac{1}{n}\int_{A_{\epsilon,R}}\frac{w'\left(\left|x\right|\right)}{\left|x\right|}\left|x\right|^{2}e^{-w\left(\left|x\right|\right)}\dd x\label{eq:linear-equality}\\
 & =\frac{c_{n}}{n}\int_{\epsilon}^{R}w'(r)r^{n}e^{-w(r)}\dd r=-\frac{c_{n}}{n}\int_{\epsilon}^{R}r^{n}\left(e^{-w(r)}\right)^{\prime}\dd r \nonumber \\
 &=\frac{c_{n}}{n}\cdot\left(\epsilon^{n}e^{-w(\epsilon)}-R^{n}e^{-w(R)}\right)+c_{n}\int_{\epsilon}^{R}r^{n-1}e^{-w(r)}\dd r\nonumber \\
 & =\frac{c_{n}}{n}\cdot\left(\epsilon^{n}e^{-w(\epsilon)}-R^{n}e^{-w(R)}\right)+\int_{A_{\epsilon,R}}\left|\nabla h\right|^{2}\dd\nu.\nonumber 
\end{align}
Since the integrands are nonnegative the integrals $\int_{A_{\epsilon,R}}\frac{w'\left(\left|x\right|\right)}{\left|x\right|}h^{2}\dd\nu$
and $\int_{A_{\epsilon,R}}\left|\nabla h\right|^{2}\dd\nu$ have a
limit when $\epsilon\to0^{+}$ and $R\to\infty$, and the limits are
finite since $\int\frac{w'\left(\left|x\right|\right)}{\left|x\right|}h^{2}\dd\nu\le\int|\nabla h|^{2}\dd\nu<\infty$.
Therefore the limits $\lim_{\epsilon\to0^{+}}\epsilon^{n}e^{-w(\epsilon)}$
and $\lim_{R\to\infty}R^{n}e^{-w(R)}$ also exist. Since $\int_{0}^{\infty}r^{n-1}e^{-w(r)}\dd r=\frac{1}{c_{n}}\int e^{-w(\left|x\right|)}\dd x<\infty$,
both of these limits have to be 0. We may therefore let $\epsilon\to0^{+}$,
$R\to\infty$ in (\ref{eq:linear-equality}) and deduce that 
\[
\int\frac{w'\left(\left|x\right|\right)}{\left|x\right|}h^{2}\dd\nu=\int\left|\nabla h\right|^{2}\dd\nu.
\]
 as claimed. 

\section{\label{sec:BL}Improved Brascamp-Lieb inequality}

In this section, we give the proof of Theorem~\ref{thm:BL-even}.
So we are working with a probability measure $\nu$ whose density
is of the form $e^{-W(x)-v(x)}$ with $W(x)=w(|x|)$ where $w$ is
smooth and satisfies the assumptions of the theorem, and $v$ is an
arbitrary even convex function on $\RR^{n}$ with values on $\RR\cup\left\{ +\infty\right\} .$
In the applications, $e^{-v}$ will be the indicator of a symmetric
convex set. But by approximation, we can easily assume that $v$ is
finite and smooth on $\RR^{n}.$

The classical Hörmander–Brascamp–Lieb inequality states that for a
smooth integrable $f$ one has
\begin{align*}
\var_{\nu}f & \le\int\left\langle \left(\nabla^{2}W+\nabla^{2}v\right)^{-1}\nabla f,\nabla f\right\rangle \dd\nu\\
 & \le\int\left\langle \left(\nabla^{2}W\right)^{-1}\nabla f,\nabla f\right\rangle \dd\nu
\end{align*}

Since $\frac{w'\left(\left|x\right|\right)}{\left|x\right|}\cdot {\rm Id}\succeq0$,
the conclusion of Theorem \ref{thm:BL-even} is clearly stronger than
this last inequality, but of course we are assuming that $f$ is even.
Recall that $f(x)=\left\langle (\nabla W+\nabla v)(x),\theta\right\rangle $
 is an equality case in the first inequality, but this function is
odd in our case. 

A direct computation of $\nabla^{2}W$ shows that for every $x\neq0,$
\[
\nabla^{2}W(x)=w''\left(\left|x\right|\right)\frac{x}{\left|x\right|}\otimes\frac{x}{\left|x\right|}+\frac{w'\left(\left|x\right|\right)}{\left|x\right|}\cdot\left({\rm Id}-\frac{x}{\left|x\right|}\otimes\frac{x}{\left|x\right|}\right),
\]
 so one can write 
\begin{equation}
\nabla^{2}W(x)+\frac{w'\left(\left|x\right|\right)}{\left|x\right|}\cdot {\rm Id}=\left(w''\left(\left|x\right|\right)+\frac{w'\left(\left|x\right|\right)}{\left|x\right|}\right)\frac{x}{\left|x\right|}\otimes\frac{x}{\left|x\right|}+2\frac{w'\left(\left|x\right|\right)}{\left|x\right|}\cdot\left({\rm Id}-\frac{x}{\left|x\right|}\otimes\frac{x}{\left|x\right|}\right).\label{eq:matrix-explicit}
\end{equation}
The condition that $t\mapsto w(e^{t})$ is convex implies that $w''(s)+\frac{w'(s)}{s}\ge0$
for all $s>0$. Hence the expression above shows that $\nabla^{2}W(x)+\frac{w'\left(\left|x\right|\right)}{\left|x\right|}\cdot {\rm Id}\succeq0$,
which we will use in the proof. 

\begin{rem}\label{rem:dual}In the statement of Theorem~\ref{thm:BL-even} and in its proof below, we encounter  expression like $\langle\big(\nabla^{2}W+\frac{w'\left(\left|x\right|\right)}{\left|x\right|}\cdot {\rm Id}\big)^{-1}a, a\rangle $ for some $a\in \RR^n$. When the matrix is singular,  one should rather use the polar form  $Q_x^\circ(a)=\sup\{\langle a, b\rangle^2\; : \; Q_x(b)\le 1\}\in[0,+\infty]$  of the quadratic form $b\mapsto Q_x(b):=\langle\big(\nabla^{2}W+\frac{w'\left(\left|x\right|\right)}{\left|x\right|}\cdot {\rm Id}\big)b, b\rangle$. Indeed, the only property we need is that $\frac12 Q_x^\circ(a)+\frac12 Q_x(b)\ge \langle a, b\rangle$, for all $a,b\in \RR^n$.
\end{rem}

\begin{proof}[Proof of Theorem \ref{thm:BL-even}]
With the notation above, consider the even function $V:=W+v$ so that $\frac{\dd\nu}{\dd x}=e^{-V}$;
we already mentioned that by approximation $v$ can be assumed to
be $C^{2}$-smooth so $V$ is $C^{2}$ as well. Since $\nu$ is
log-concave with respect to $e^{-W}$ it follows that $\nabla^{2}V\succeq\nabla^{2}W$
as positive definite matrices. Also write $A(x)=\nabla^{2}W(x)+\frac{w'\left(\left|x\right|\right)}{\left|x\right|}\cdot {\rm Id}.$

Consider the operator $Lu=\Delta u-\langle\nabla V,\nabla u\rangle$,
that is the Laplace operator $\nabla^{\ast}\nabla$ on $L^{2}(\nu)$.
We are given an even function $f$. We can add a constant to $f$
and assume without loss of generality that $\int f\dd\nu=0$. It is
well known then that $f$ can be approximated in $L^2(\nu)$ by functions of the %%R1, comment 12
form $Lu$ for smooth compactly supported $u$ (see for instance \cite{Cordero-Erausquin2004}).
Moreover, since $V$ is even and $f$ is even we can also assume that
$u$ is even. Therefore it is enough to prove 
\[
\int\left(\left(Lu-f\right)^{2}-f^{2}+\left\langle A^{-1}\nabla f,\nabla f\right\rangle \right)\dd\nu\ge0,
\]
 that is
\[
\int\left(\left(Lu\right)^{2}-2Lu\cdot f+\left\langle A^{-1}\nabla f,\nabla f\right\rangle \right)\dd\nu\ge0.
\]
 Integrating by parts we see that 
\begin{align*}
\int Lu\cdot f\dd\nu & =-\int\left\langle \nabla u,\nabla f\right\rangle \dd\nu,\\
\int\left(Lu\right)^{2}\dd\nu & =\int\left(\left\Vert \nabla^{2}u\right\Vert _{2}^{2}+\left\langle \left(\nabla^{2}V\right)\cdot\nabla u,\nabla u\right\rangle \right)\dd\nu\\
 & \ge\int\left(\left\Vert \nabla^{2}u\right\Vert _{2}^{2}+\left\langle \left(\nabla^{2}W\right)\cdot\nabla u,\nabla u\right\rangle \right)\dd\nu,
\end{align*}
 where $\left\Vert A\right\Vert _{2}=\sqrt{\tr\left(AA^{\ast}\right)}$ is
the Hilbert-Schmidt norm. Therefore it is enough to prove the inequality
\[
\int\left(\left\Vert \nabla^{2}u\right\Vert _{2}^{2}+\left\langle \left(\nabla^{2}W\right)\nabla u,\nabla u\right\rangle +2\left\langle \nabla u,\nabla f\right\rangle +\left\langle A^{-1}\nabla f,\nabla f\right\rangle \right)\dd\nu\ge0.
\]
We have the pointwise identity
\[
\left|A^{-\frac{1}{2}}\nabla f+A^{\frac{1}{2}}\nabla u\right|^{2}=\left\langle A^{-1}\nabla f,\nabla f\right\rangle +2\left\langle \nabla f,\nabla u\right\rangle +\left\langle A\cdot\nabla u,\nabla u\right\rangle ,
\]
so our goal can be written as
\[ %% R1, comment 13
\int\left(\left\Vert \nabla^{2}u\right\Vert _{2}^{2}+\left\langle \left(\nabla^{2}W\right)\nabla u,\nabla u\right\rangle +\left|A^{-\frac{1}{2}}\nabla f+A^{\frac{1}{2}}\nabla u\right|^2-\left\langle A\cdot\nabla u,\nabla u\right\rangle \right)\dd\nu\ge0.
\]
 As $\left|A^{-\frac{1}{2}}\nabla f+A^{\frac{1}{2}}\nabla u\right|^2\ge0$ (this corresponds to the duality relation recalled in Remark~\ref{rem:dual} above), it is therefore
enough to prove that 
\begin{equation} %%R1, comment 14
\int\frac{w'\left(\left|x\right|\right)}{\left|x\right|}\left|\nabla u\right|^{2}\dd\nu=\int\left\langle \left(A-\nabla^{2}W\right)\cdot\nabla u,\nabla u\right\rangle \dd\nu\le\int\left\Vert \nabla^{2}u\right\Vert _{2}^{2}\dd\nu.\label{eq:BL-goal}
\end{equation}
But this follows form Theorem \ref{thm:weighted-poincare}: Every
derivative $h_{i}=\frac{\partial u}{\partial x_{i}}$ is odd, so by
the Theorem \ref{thm:weighted-poincare} we have
\[
\int\frac{w'\left(\left|x\right|\right)}{\left|x\right|}h_{i}^{2}\dd\nu\le\int\left|\nabla h_{i}\right|^{2}\dd\nu.
\]
 Summing over $1\le i\le n$ we obtain the desired inequality (\ref{eq:BL-goal}). 
\end{proof}

\section{\label{sec:the-B-property}The (B) property}

In this section we prove Theorem \ref{thm:B-result} in the functional
form (\ref{eq:functional_B_property}):
\begin{proof}[Proof of Theorem \ref{thm:B-result}]
By approximation we may assume $w$ is well defined on $[0,\infty)$
and $C^{2}$-smooth there with $w'' > 0$. Write $W(x):=w(|x|)$. 

As we said, we will prove the more general form (\ref{eq:functional_B_property}).
So we fix an arbitrary \emph{even} convex function $v:\RR^{n}\to\RR\cup\left\{ +\infty\right\} $,
and our goal is to prove that 
\begin{equation}
t\mapsto\int_{\RR^{n}}e^{-v(e^{tA}y)-W(y)}\dd y=e^{-t\cdot\tr A}\cdot\int_{\RR^{n}}e^{-v\left(x\right)-W\left(e^{-tA}x\right)}\dd x\label{eq:change-of-vars}
\end{equation}
is log-concave. Since the function $t\mapsto e^{-t\cdot\tr A}$ is
clearly log-linear, and since the change of variables $t\mapsto-t$
preserves log-concavity, we are led to prove that the function 
\[
\rho_{v}(t)=\int_{\RR^{n}}e^{-v\left(x\right)-W\left(e^{tA}x\right)}\dd x
\]
 is log-concave. 

To do so we need to show that $\rho_{v}(t)\rho_{v}''(t)\le\rho_{v}'(t)^{2}$ %%R1, comment 15
for every $t\in\RR$. However from the change of variables (\ref{eq:change-of-vars})
one can check that $\rho_{v}(t+h)=\rho_{\tilde{v}}(h)$ where $\tilde{v}(x)=v(e^{-tA}x)+t\cdot\tr A$
is again an even convex function. So we see it is enough to show that
$\rho(0)\rho''(0)\le\rho'(0)^{2}$ for $\rho:=\rho_{v}$ and $v$ is %%R1, comment 16
an arbitrary even convex function. 

Computing these derivatives we obtain 
\begin{align*}
\rho^{\prime}(t) & =-\int e^{-W(e^{tA}x)}\left\langle \nabla W(e^{tA}x),Ae^{tA}x\right\rangle e^{-v(x)}\,\dd x\\
\rho^{\prime\prime}(t) & =\int e^{-W\left(e^{tA}x\right)}\left\langle \nabla W(e^{tA}x),Ae^{tA}x\right\rangle ^{2}e^{-v(x)}\,\dd x\\
 & \qquad-\int e^{-W\left(e^{tA}x\right)}\left(\left\langle \nabla^{2}W(e^{tA}x)\cdot Ae^{tA}x,Ae^{tA}x\right\rangle +\left\langle \nabla W\left(e^{tA}x\right),A^{2}e^{tA}x\right\rangle \right)e^{-v(x)}\,\dd x,
\end{align*}
 so the condition $\rho(0)\rho''(0)\le\rho'(0)^{2}$ becomes %%R1, comment 16
\begin{align*}
\int e^{-v(x)}\,\dd\mu(x) & \cdot\int\left(\left\langle \nabla W,Ax\right\rangle ^{2}-\left\langle \nabla^{2}W\cdot Ax,Ax\right\rangle -\left\langle \nabla W,A^{2}x\right\rangle \right)e^{-v(x)}\dd\mu(x)\\
 & \le\left(\int\left\langle \nabla W,Ax\right\rangle e^{-v(x)}\,\dd\mu(x)\right)^{2}.
\end{align*}
 Introduce the probability measure 
\[
\dd\nu(x)=\frac{e^{-v(x)}}{\int e^{-v}\,\dd\mu}\,\dd\mu(x).
\]
Our aim is to prove that
\[
\int\left(\left\langle \nabla W,Ax\right\rangle ^{2}-\left\langle \nabla^{2}W\cdot Ax,Ax\right\rangle -\left\langle \nabla W,A^{2}x\right\rangle \right)\dd\nu\le\left(\int\left\langle \nabla W,Ax\right\rangle \dd\nu\right)^{2},
\]
 that is
\begin{equation}
\int\left\langle \nabla W,Ax\right\rangle ^{2}\dd\nu-\left(\int\left\langle \nabla W,Ax\right\rangle \dd\nu\right)^{2}\le\int\left(\left\langle \nabla^{2}W\cdot Ax,Ax\right\rangle +\left\langle \nabla W,A^{2}x\right\rangle \right)\dd\nu.\label{eq:local-B-ineq}
\end{equation}
Actually, this aimed inequality is really equivalent to the strong
(B) inequality for $\mu.$ We claim that this inequality follows from
Theorem \ref{thm:BL-even} for the function 
\[
f_{0}(x):=\left\langle \nabla W(x),Ax\right\rangle .
\]
Indeed, note first that Theorem \ref{thm:BL-even} is applicable since
$\nu$ is even and log-concave with respect to $\mu$ and $f_{0}$
is even. Next we have to interpret correctly the right-hand side of
(\ref{eq:local-B-ineq}). Note that 
\[
\nabla f_{0}=\nabla^{2}W\cdot Ax+A\cdot\nabla W
\]
and so 
\[
\left\langle \nabla^{2}W\cdot Ax,Ax\right\rangle +\left\langle \nabla W,A^{2}x\right\rangle =\left\langle \nabla f_{0},Ax\right\rangle 
\]
where we have used the symmetry of $A$. This means that (\ref{eq:local-B-ineq}) can be written as
\[
\textrm{Var}_{\nu}(f_{0})\le\int\langle\nabla f_{0},Ax\rangle\,\dd\nu(x).
\]
But since $\nabla W(x)=\frac{w'\left(\left|x\right|\right)}{\left|x\right|}x$
we can write 
\begin{equation}
\nabla f_{0}=\left(\nabla^{2}W+\frac{w'\left(\left|x\right|\right)}{\left|x\right|}{\rm Id}\right)Ax\label{eq:relation}
\end{equation}
which implies that 
\[
\left\langle \nabla f_{0},Ax\right\rangle =\left\langle \left(\nabla^{2}W+\frac{w'\left(\left|x\right|\right)}{\left|x\right|}{\rm Id}\right)^{-1}\nabla f_{0},\nabla f_{0}\right\rangle .
\]
 This shows that (\ref{eq:local-B-ineq}) follows from Theorem
\ref{thm:BL-even} as claimed. 
\end{proof}
There is a hidden but crucial choice behind the apparently trivial
relation (\ref{eq:relation}). Indeed, let us consider the simple
case where $w(t)=t^{p}/p$ (so $W(x)=|x|^{p}/p)$ and $A={\rm Id}$.
The equivalent formulation (\ref{eq:local-B-ineq}) of the (B) inequality
is then 
\[
\textrm{Var}_{\nu}(f_{0})=\textrm{Var}_{\nu}(|\cdot|^{p})\le p\int|x|^{p}\,\dd\nu(x)=\int\langle\nabla f_{0},x\rangle\,\dd\nu
\]
for $f_{0}(x)=|x|^{p}$ and $\nu$ an even probability measure which is log-concave
with respect to $e^{-W(x)}\dd x=e^{-|x|^{p}/p}\,\dd x.$ We see that there %%R1, comment 17
are several possible interpretations of the last term, since we have
both
\[
x=\frac{p-1}{p}(\nabla^{2}W)^{-1}\cdot\nabla f_{0}\qquad\textrm{and \ensuremath{\qquad x=}\ensuremath{\frac{1}{p|x|^{p-2}}\,\nabla f_{0}.}}
\]
By the way, here we can invoke homogeneity to shorten computations
since $f_{0}=pW$. Every choice of a matrix-valued function $B(x)$
such that $B\cdot x=\nabla f_{0}$ leads to the natural question whether
the corresponding Brascamp-Lieb type inequality $\textrm{Var}_{\nu}(f)\le\int\langle B^{-1}\nabla f,\nabla f\rangle\dd\nu$
holds for every smooth even $f$; This would imply  the (B) inequality.
The two formulas above coincide in the case of the Gaussian measure
($p=2)$, but not in general. Our choice (\ref{eq:relation}) is some
combination of the two:
\[
x=\left(\nabla^{2}W+|x|^{p-2}{\rm Id}\right)^{-1}\nabla f_{0}(x).
\]

\section{Dimensional Brunn-Minkowski inequality}

In this section we prove Theorem \ref{thm:GZ}, that is the Gardner-Zvavitch
conjecture for rotationally invariant measures. We heavily rely on
the computations done by Kolesnikov and Livshyts, and on the ideas
introduced by Eskenazis and Moschidis in their solution of the Gaussian
case. 

As for the Brunn-Minkowski inequality and the (B) inequality, we can
prove this $\frac{1}{n}$-concavity by computing the second derivative
in the parameter $\lambda.$ This was done by Kolesnikov and Livshyts
\cite{Kolesnikov2021} (see Lemma 2.3), who found the following neat
\emph{sufficient} condition; this result is somehow a substitute to the duality argument used in the proof of Theorem~\ref{thm:BL-even}.  Below, the notation $\mu_{K}$ refers
to the normalized restriction of a measure $\mu$ to a set $K$ with
$\mu(K)<\infty$. 
\begin{thm}
\label{thm:kolesnikov-livshyts}Let $\mu$ be a locally finite measure
on $\RR^{n}$ with density $e^{-W}$. Assume that for every symmetric
convex body $K\subset\RR^{n}$ and every smooth even function $u:K\to\RR$
with $Lu:=\Delta u-\langle\nabla W,\nabla u\rangle\equiv1$ in $K$
we have 
\[
\int\left(\left\Vert \nabla^{2}u\right\Vert _{2}^{2}+\left\langle \nabla^{2}W\cdot\nabla u,\nabla u\right\rangle \right)\dd\mu_{K}\ge\frac{1}{n}.
\]
 Then for every symmetric convex bodies $K,L\subseteq\RR^{n}$ and
every $0 \le \lambda \le 1$ we have %%R1, comment 18
\[
\mu\left((1-\lambda)K+\lambda L\right)^{\frac{1}{n}}\ge(1-\lambda)\mu(K)^{\frac{1}{n}}+\lambda\mu(L)^{\frac{1}{n}}.
\]
 
\end{thm}

Let us mention that this formulation builds upon previous ideas introduced by Kolesnikov and Milman
in \cite{Kolesnikov2018} and \cite{Kolesnikov2022}, in particular
the idea to obtain Poincaré type inequalities on the \emph{boundary}
of the given domain $K$ by expressing a function on $\partial K$
as a Neumann data of a function in the interior of $K$.

We now establish the condition in Theorem~\ref{thm:kolesnikov-livshyts} using our weighted Poincaré inequality (Theorem~\ref{thm:weighted-poincare}). Loosely speaking, we have "Theorem~\ref{thm:weighted-poincare} $\Rightarrow$ Theorem~\ref{thm:kolesnikov-livshyts} $\Rightarrow$Brunn-Minkowski", whereas previously we had "Theorem~\ref{thm:weighted-poincare} $\Rightarrow$ Theorem~\ref{thm:BL-even} $ \Rightarrow$ (B) inequality", with the difference that the formulation of Theorem~\ref{thm:BL-even} was specific to the rotationally invariant case and possibly of independent interest. 
%in hand, we can proceed
%with the proof of our result. 
%
\begin{proof}[Proof of Theorem \ref{thm:GZ}]
Write $W(x)=w\left(\left|x\right|\right)$ and assume by approximation
that $w$ is smooth. We begin by following the argument of Eskenazis
and Moschidis \cite{Eskenazis2021}. 

Define $r=\frac{\left|x\right|^{2}}{2n}$ and note that
\[
\left\Vert \nabla^{2}u\right\Vert _{2}^{2}=\left\Vert \nabla^{2}(u-r)\right\Vert _{2}^{2}+\frac{2}{n}\Delta u-\frac{1}{n}.
\]
 Since $Lu=1$ we have $\Delta u=\left\langle \nabla W,\nabla u\right\rangle +1$,
so 
\begin{equation}
\left\Vert \nabla^{2}u\right\Vert _{2}^{2}=\left\Vert \nabla^{2}(u-r)\right\Vert _{2}^{2}+\frac{2}{n}\left\langle \nabla W,\nabla u\right\rangle +\frac{1}{n}.\label{eq:r-decompose}
\end{equation}

Next, we apply our weighted Poincaré inequality (Theorem \ref{thm:weighted-poincare}).
Since $K$ is symmetric the measure $\mu_{K}$ is even, and of course
it is log-concave with respect to $\mu$. Moreover, for every $i$
the derivative $\partial_{i}(u-r)$ is odd, since $u-r$ is even.
Hence 
\[
\int\left|\nabla\partial_{i}(u-r)\right|^{2}\dd\mu_{K}\ge\int\frac{w'\left(\left|x\right|\right)}{\left|x\right|}\left(\partial_{i}(u-r)\right)^{2}\dd\mu_{K}.
\]
 Summing over $1\le i\le n$ we get 
\begin{align*}
\int\left\Vert \nabla^{2}(u-r)\right\Vert _{2}^{2}\dd\mu_{K} & \ge\int\frac{w'\left(\left|x\right|\right)}{\left|x\right|}\left|\nabla(u-r)\right|^{2}\dd\mu_{K}\\
 & =\int\frac{w'\left(\left|x\right|\right)}{\left|x\right|}\left(\left|\nabla u\right|^{2}-2\left\langle \nabla u,\nabla r\right\rangle +\left|\nabla r\right|^{2}\right)\dd\mu_{K}\\
 & \ge\int\frac{w'\left(\left|x\right|\right)}{\left|x\right|}\left(\left|\nabla u\right|^{2}-\frac{2}{n}\left\langle \nabla u,x\right\rangle \right)\dd\mu_{K}\\
 & =\int\left(\frac{w'\left(\left|x\right|\right)}{\left|x\right|}\left|\nabla u\right|^{2}-\frac{2}{n}\left\langle \nabla W,\nabla u\right\rangle \right)\dd\mu_{K},
\end{align*}
 where in the last equality we used the fact that $\nabla W(x)=\frac{w'\left(\left|x\right|\right)}{\left|x\right|}x$.
Therefore using (\ref{eq:r-decompose}) we obtain 
\begin{align*}
\int\left(\left\Vert \nabla^{2}u\right\Vert _{2}^{2}+\left\langle \nabla^{2}W\cdot\nabla u,\nabla u\right\rangle \right)\dd\mu_{K} & \ge\int\left\langle \left(\nabla^{2}W+\frac{w'\left(\left|x\right|\right)}{\left|x\right|}{\rm Id}\right)\nabla u,\nabla u\right\rangle \dd\mu_{K}+\frac{1}{n}\\
 & \ge0+\frac{1}{n}=\frac{1}{n}.
\end{align*}
 The last inequality is true since our assumption on $w$ implies
that the matrix $\left(\nabla^{2}W+\frac{w'\left(\left|x\right|\right)}{\left|x\right|}{\rm Id}\right)$
is nonnegative (see the computation in the beginning of Section 3). 

By Theorem \ref{thm:kolesnikov-livshyts} we conclude that 
\[
\mu\left((1-\lambda)K+\lambda L\right)^{\frac{1}{n}}\ge(1-\lambda)\mu(K)^{\frac{1}{n}}+\lambda\mu(L)^{\frac{1}{n}}.
\]
 for all symmetric $K,L$ and all $0 \le \lambda \le 1$. %R1, comment 18
\end{proof}

\section{\label{sec:mixtures}Mixtures}

As was mentioned in the introduction, before the results of this paper
there were very few examples of measures known to have the (B) property
except the Gaussian measure. The only such examples we are aware of
in dimension $n\ge3$ come from a result of Eskenazis, Nayar and Tkocz
(\cite{Eskenazis2018}) about Gaussian mixtures. We will now briefly
explain and slightly extend their result. 
\begin{prop}
\label{prop:mixtures}Let $X=\left(X_{1},X_{2},\ldots,X_{n}\right)$
be a random vector with a probability density on $\RR^{n}$ which
is rotationally invariant and log-concave. Let $Y=(Y_{1},Y_{2},\ldots,Y_{n})$
be a random random vector on $(\RR^{+})^{n}$ independent of $X$
with probability density $h:(0,\infty)^{n}\to\RR$ such that $(s_{1},s_{2},\ldots,s_{n})\mapsto h\left(e^{s_{1}},e^{s_{2}},\ldots,e^{s_{n}}\right)$
is log-concave. Let $\nu$ denote the distribution of $\left(X_{1}Y_{1},X_{2}Y_{2},\ldots,X_{n}Y_{n}\right)$.
Then for every symmetric convex body $K\subseteq\RR^{n}$ the function
\[
(t_{1},t_{2},\ldots,t_{n})\mapsto\nu(e^{\Delta\left(t_{1},\ldots,t_{n}\right)}K)
\]
 is log-concave on $\RR^{n}$; In particular $t\mapsto\nu\left(e^{t}K\right)$
is log-concave on $\RR$. 
\end{prop}

Here we use the notation $\Delta(t_{1},\ldots,t_{n})$ for the diagonal
matrix with entries $t_{1},t_{2},\ldots,t_{n}$ on its diagonal. 
\begin{proof}
For every Borel set $K\subseteq\RR^{n}$ we have 
\begin{align*}
\nu(K) & =\PP\left(\left(X_{1}Y_{1},X_{2}Y_{2},\ldots,X_{n}Y_{n}\right)\in K\right)\\
 & =\int_{(0,\infty)^{n}}\PP\left(\left(y_{1}X_{1},y_{2}X_{2},\ldots,y_{n}X_{n}\right)\in K\right)h(y)\dd y
\end{align*}
 We perform a change of variables $e^{-s}=y$ (i.e. $e^{-s_{i}}=y_{i}$
for $1\le i\le n$). Then $\dd y=e^{-\sum_{i=1}^{n}s_{i}}\dd s$,
so 
\begin{align*}
\nu(K) & =\int_{\RR^{n}}\PP\left(e^{-\Delta(s_{1},\ldots,s_{n})}\cdot X\in K\right)h(e^{-s})e^{-\sum s_{i}}\dd s\\
 & =\int_{\RR^{n}}\PP\left(X\in e^{\Delta(s_{1},\ldots,s_{n})}K\right)h(e^{-s})e^{-\sum s_{i}}\dd s\\
 & =\int_{\RR^{n}}\mu\left(e^{\Delta(s_{1},\ldots,s_{n})}K\right)\cdot h(e^{-s})e^{-\sum s_{i}}\dd s,
\end{align*}
 where $\mu$ denotes the distribution of $X$. 

Therefore if we now assume that $K$ is a symmetric convex body, then
\[
\nu(e^{\Delta\left(t_{1},\ldots,t_{n}\right)}K)=\int_{\RR^{n}}\mu\left(e^{\Delta(s_{1}+t_{1},\ldots,s_{n}+t_{n})}K\right)\cdot h(e^{-s})e^{-\sum s_{i}}\dd s.
\]
By Theorem \ref{thm:B-result} the function $(t,s)\mapsto\mu\left(e^{\Delta(s_{1}+t_{1},\ldots,s_{n}+t_{n})}K\right)$
is log-concave on $\RR^{2n}$, so by our assumption on $h$ 
\[
(t,s)\mapsto\mu\left(e^{\Delta(s_{1}+t_{1},\ldots,s_{n}+t_{n})}K\right)h(e^{-s})e^{-\sum s_{i}}
\]
 is also log-concave. It is a well-known corollary of the Prékopa
inequality (\ref{eq:BM0}) that marginals of log-concave functions
are also log-concave. Hence the function 
\[
t\mapsto\int_{\RR^{n}}\mu\left(e^{\Delta(s_{1}+t_{1},\ldots,s_{n}+t_{n})}K\right)\cdot h(e^{-s})e^{-\sum s_{i}}\dd s=\nu(e^{\Delta\left(t_{1},\ldots,t_{n}\right)}K)
\]
is log-concave.
\end{proof}
The result of \cite{Eskenazis2018} is identical to the proposition
above, with an identical proof, except the fact that they have to
assume $X$ is Gaussian in order to use the original result of \cite{Cordero-Erausquin2004},
while we can use instead Theorem \ref{thm:B-result}. Of course, the
assumption in the proposition that the distribution $\mu$ of $X$
is log-concave may be replaced with the weaker assumption on $\mu$
of Theorem \ref{thm:B-result}. 

Proposition \ref{prop:mixtures} is only useful if one can identify
measures $\nu$ which satisfy its assumptions. It is shown in \cite{Eskenazis2018}
that if $0<p\le1$ and if $\nu$ has density proportional to $e^{-\left\Vert x\right\Vert _{p}^{p}}=e^{-\sum\left|x_{i}\right|^{p}}$,
then $\nu$ satisfies the assumptions of the proposition (with a Gaussian
random vector $X$) and therefore has the (B) property. The same is
shown for the product measure $\nu=\nu_{1}^{\otimes n}$, where $\nu_{1}$
is the distribution of a $p$-stable random variable for $0<p\le1$.
No other examples are constructed.

Since we now have more freedom in the choice of $X$, our proposition
applies to more measures $\nu$ than the theorem of \cite{Eskenazis2018}.
However, at the moment we don't have any natural measure to propose
that could be handled using this extra freedom.

\begin{funding}
The second author is partially supported by ISF grant 1468/19 and BSF grant 2016050.
\end{funding}

%% or include bibliography directly:

\end{document}